\documentclass[11pt,a4paper]{article}
\usepackage{amsmath,amssymb,amsthm,bbm, mathtools}
\usepackage{mwe}
\usepackage[T1]{fontenc}
\usepackage[applemac]{inputenc}	
\usepackage[english]{babel}
\usepackage{geometry, color, graphicx, amsfonts,amsmath,dsfont,amssymb,stmaryrd,bbm,graphicx,pgfplots,calc}

\newtheorem{theorem}{Theorem}
\newtheorem{corollary}[theorem]{Corollary}
\newtheorem{lemma}[theorem]{Lemma}
\newtheorem{proposition}[theorem]{Proposition}

\newtheorem{definition}[theorem]{Definition} 
\newtheorem{observation}[theorem]{Observation} \theoremstyle{remark}
\newtheorem{remark}[theorem]{Remark}

\newcommand{\comment}[1]{}

\newcommand{\R}{\mathbb R}
\newcommand{\N}{\mathbb N}

\newcommand{\EE}{\mathbb{E}}

\newcommand{\eps}{\varepsilon}

\newcommand{\supp}{{\rm supp}}
\newcommand{\Vol}{{\rm Vol}}

\providecommand{\Prob}[1]{\mathbb{P}\left(#1\right)}

 \newgeometry{tmargin=2cm, bmargin=2cm, lmargin=2cm, rmargin=2cm}

\pgfplotsset{compat=1.14}

\begin{document}
\title{Further investigations of R\'enyi entropy power inequalities and an entropic characterization of s-concave densities}
\author{Jiange Li, Arnaud Marsiglietti, James Melbourne
}
\date{}
\maketitle

\begin{abstract}
We investigate the role of convexity in R\'enyi entropy power inequalities. After proving that a general R\'enyi entropy power inequality in the style of Bobkov-Chistyakov (2015) fails when the R\'enyi parameter $r\in(0,1)$, we show that random vectors with $s$-concave densities do satisfy such a R\'enyi entropy power inequality. Along the way, we establish the convergence in the Central Limit Theorem for R\'enyi entropies of order $r\in(0,1)$ for log-concave densities and for compactly supported, spherically symmetric and unimodal densities, complementing a celebrated result of Barron (1986). Additionally, we give an entropic characterization of the class of $s$-concave densities, which extends a classical result of Cover and Zhang (1994).
\end{abstract}

\section{Introduction}

Let $X$ be a random vector in $\R^d$. Suppose that $X$ has the density $f$ with respect to the Lebesgue measure. For $r \in (0,1) \cup (1,\infty)$, the R\'enyi entropy of order $r$ (or simply, $r$-R\'enyi entropy) is defined as
\begin{equation}\label{eq:r-renyi-entropy}
 h_r(X) = \frac{1}{1-r}\log \int_{\R^d} f(x)^r dx.
\end{equation}
For $r\in\{0, 1, \infty\}$, the $r$-R\'enyi entropy can be extended continuously such that the RHS of \eqref{eq:r-renyi-entropy} is
$\log |\supp(f)|$ for $r=0$; $-\int_{\R^d}f(x)\log f(x)dx$ for $r=1$; and $- \log \|f\|_\infty$ for $r=\infty$.
The case $r=1$ corresponds to the classical Shannon differential entropy. Here, we denote by $|\supp(f)|$ the Lebesgue measure of the support of $f$, and $\|f\|_\infty$ represents the essential supremum of $f$.  The $r$-R\'enyi entropy power is defined by
$$
 N_r(X) = e^{2 h_r(X)/d}.
$$
In the following, we drop the subscript when $r=1$. 

The classical Entropy Power Inequality (henceforth, EPI) of Shannon \cite{Sha48} and Stam \cite{Sta59}, states that the entropy power $N(X)$ is super-additive on the sum of independent random vectors. There has been recent success in obtaining
extensions of the EPI from the Shannon differential entropy to $r$-R\'enyi entropy. In \cite{BC14,BC15:1}, Bobkov and Chistyakov showed that, at the expense of an absolute constant $c>0$, the following R\'enyi EPI of order $r\in[1,\infty]$ holds 
\begin{equation}\label{eq:BC-REPI}
N_r(X_1 + \cdots + X_n) \geq c \sum_{i=1}^n N_r(X_i).
\end{equation}
Ram and Sason soon after gave a sharpened constant depending on the number of summands \cite{RS16}.
Madiman, Melbourne, and Xu sharpened constants in the $r=\infty$ case by identifying extremizers in \cite{MMX17:2,XMM17:isit:1}.  Savar\'e and Toscani \cite{ST14} showed that a modified R\'enyi entropy power was concave along the solution of a nonlinear heat equation, which generalizes Costa's concavity of entropy power \cite{Cos85b}.  Bobkov and Marsiglietti \cite{BM16} proved the following variant of R\'enyi EPI
\begin{equation}\label{eq:alpha-modifi}
N_r(X+Y)^\alpha\geq N_r(X)^\alpha+N_r(Y)^\alpha
\end{equation}
for $r>1$ and some exponent $\alpha$ only depending on $r$. It is clear that \eqref{eq:alpha-modifi} holds for more than two summands. Improvement of the exponent $\alpha$ was given by Li \cite{Li17}. 

One of our goals is to establish analogues of \eqref{eq:BC-REPI} and \eqref{eq:alpha-modifi} when the R\'enyi parameter $r\in(0, 1)$. Both \eqref{eq:BC-REPI} and \eqref{eq:alpha-modifi} can be derived from Young's convolution inequality in conjunction with the entropic comparison inequality $h_{r_1}(X)\geq h_{r_2}(X)$ for any $0\leq r_1\leq r_2$. The latter is an immediate consequence of Jensen's inequality.  When  the R\'enyi parameter $r\in(0, 1)$, analogues of \eqref{eq:BC-REPI} and \eqref{eq:alpha-modifi} require a converse of the entropic comparison inequality aforementioned. This technical issue prevents a general R\'enyi EPI of order $r\in(0, 1)$ for generic random vectors. Our first result shows that a general R\'enyi EPI of the form \eqref{eq:BC-REPI} indeed fails for all $r \in (0,1)$. 
\begin{theorem} \label{thm: REPI failure}
For any $r \in (0,1)$ and $\varepsilon >0$, there exist independent random vectors $X_1, \cdots, X_n$ in $\R^d$, for some $d \geq 1$ and $n \geq 2$, such that
\begin{equation}\label{eq:failure}
N_r(X_1 + \cdots + X_n) < \varepsilon \sum_{i=1}^n N_r(X_i).
\end{equation}
\end{theorem}

We have an explicit construction of such random vectors. They are essentially truncations of some spherically symmetric random vectors with finite covariance matrices and infinite R\'enyi entropies of order $r\in(0, 1)$. The key point is the convergence along the Central Limit Theorem (henceforth, CLT) for R\'enyi entropies of order $r\in(0, 1)$; that is, the $r$-R\'enyi entropy of 
their normalized sum 
converges to the $r$-R\'enyi entropy of a Gaussian. This implies that, after appropriate normalization, the LHS of \eqref{eq:failure} is finite, but the RHS of \eqref{eq:failure} can be as large as possible. The entropic CLT has been studied for a long time. A celebrated result of Barron \cite{Bar86} shows the convergence in the CLT for Shannon entropy (see \cite{Joh04:book} for a multidimensional setting). The recent work of Bobkov and Marsiglietti \cite{BM18} studies the convergence in the CLT for R\'enyi entropy of order $r>1$ for real-valued random variables (see also \cite{BCG16:2} for convergence in R\'enyi divergence, which is not equivalent to convergence in R\'enyi entropy unless $r=1$). In Section 2, we establish the analogue of \cite[Theorem 1.1]{BM18} in higher dimensions and we prove convergence along the CLT for R\'enyi entropies of order $r\in(0,1)$ for a large class of densities.

As mentioned above, the reverse entropic comparison inequality prevents R\'enyi EPIs of order $r\in(0, 1)$ for generic random vectors. However, a large class of random vectors with the so-called $s$-concave densities do satisfy such a reverse entropic comparison inequality. Our next results show that R\'enyi EPI of order $r\in(0, 1)$ holds for such densities. This extends the earlier work of Marsiglietti and Melbourne \cite{MM18-IEEE,MarsigliettiMelbourne:isit} for log-concave densities (which corresponds to the $s=0$ case). 

Let $s \in [-\infty, \infty]$. A function $f \colon \mathbb{R}^d \to [0,\infty)$ is called $s$-concave if the inequality
\begin{equation}\label{eq:s-concave}
f((1-\lambda)x+\lambda y) \geq ((1-\lambda)f(x)^s+\lambda f(y)^s)^{1/s}
\end{equation}
holds for all $x,y \in \mathbb{R}^d$ such that $f(x)f(y)>0$ and $\lambda \in (0,1)$.
For $s \in \{-\infty,0,\infty\}$, the RHS of \eqref{eq:s-concave} is understood in the limiting sense; that is $\min \{f(x), f(y)\}$ for $s=-\infty$, $f(x)^{1-\lambda}f(y)^\lambda$ for $s=0$, and $\max \{f(x), f(y)\}$ for $s=\infty$.
The case $s=0$ corresponds to log-concave functions.  The study of measures with $s$-concave densities was initiated by Borell in the seminal work \cite{Borell74,Bor75a}. One can think of $s$-concave densities, in particular log-concave densities, as functional versions of convex sets. There has been a recent stream of research on a formal parallel relation between functional inequalities of $s$-concave densities and geometric inequalities of convex sets.

\begin{theorem} \label{thm: s-concave REPI}
For any $s \in (-1/d,0)$ and $r \in (-sd,1)$, there exists $c = c(s,r,d,n)$ such that for all independent random vectors $X_1, \cdots, X_n$ with $s$-concave densities in $\R^d$, we have
$$
N_r(X_1 + \cdots + X_n) \geq c \sum_{i=1}^n N_r(X_i).
$$
In particular, one can take
$$
c = r^{\frac{1}{1-r}} \left( 1 + \frac{1}{n|r'|} \right)^{1+n|r'|} \left(\prod_{k=1}^d \frac{(1+ks)^{|r'|(n-1)} (1+\frac{ks}{r})^{1+|r'|}}{(1+ks(1+\frac{1}{n|r'|}))^{1+n|r'|}} \right)^{\frac{2}{d}},
$$
where $r'=r/(r-1)$ is the H\"older conjugate of $r$.
\end{theorem}


\begin{theorem}\label{thm: alpha-modif}
Given $s\in(-1/d, 0)$, there exist $0<r_0<1$ and $\alpha= \alpha(s,r,d)$ such that for $r\in(r_0, 1)$ and independent random vectors $X$ and $Y$ in $\R^d$ with $s$-concave densities,
$$
N_r(X+Y)^{\alpha}\geq N_r(X)^{\alpha}+N_r(Y)^{\alpha}.
$$
In particular, one can take
\begin{eqnarray*}
r_0 &=& \left(1-\frac{2}{1+\sqrt{3}}\Big(1+\frac{1}{sd}\Big)\right)^{-1}\\
\alpha &=& \left( 1 + \frac{\log r + (r+1)\log \frac{r+1}{2r} + C(s) }{(1-r) \log 2} \right)^{-1},
\end{eqnarray*}
where
$$
C(s) = \frac{2}{d} \sum_{k=1}^d \left( \log\left(1+\frac{ks}{r}\right) + r\log(1+ks) - (r+1)\log\left(1 + \frac{ks(r+1)}{2r}\right) \right).
$$
\end{theorem}

Owing to the convexity, random vectors with $s$-concave densities also satisfy a reverse EPI, which was first proved by Bobkov and Madiman \cite{BM12:jfa}. This can be seen as the functional lifting of Milman's well known reverse Brunn-Minkowski inequality \cite{Mil86}. Motivated by Busemann's theorem \cite{Bus49} in convex geometry, Ball, Nayar and Tkocz \cite{BNT15} conjectured that the following reverse EPI 
\begin{equation}\label{eq:reverse epi}
N(X+Y)^{1/2}\leq N(X)^{1/2}+N(Y)^{1/2}
\end{equation}
holds for any symmetric log-concave random vector $(X, Y)\in\R^2$. The $r$-R\'enyi entropy analogue was asked in \cite{MMX17:1}, and the $r=2$ case was soon verified in \cite{Li17}. It was also observed in \cite{Li17} that the $r$-R\'enyi entropy analogue  is equivalent to the convexity of $p$-cross-section body in convex geometry introduced by Gardner and Giannopoulos \cite{GG99}. The equivalent linearization of \eqref{eq:reverse epi} reads as follows. Let $(X, Y)$ be a symmetric log-concave random vector in $\R^2$ such that $h(X)=h(Y)$. Then for any $\lambda\in[0, 1]$ we have
$$
h((1-\lambda)X+\lambda Y)\leq h(X).
$$
Cover and Zhang \cite{CZ94} proved the above inequality under the stronger assumption that $X$ and $Y$ have the same log-concave distribution. They also showed that this provides a characterization of log-concave distributions on the real line. The following theorem extends Cover and Zhang's result from log-concave densities to a more general class of $s$-concave densities. This gives an entropic characterization of $s$-concave densities and implies a reverse R\'enyi EPI for random vectors with the same $s$-concave density. 

\begin{theorem} \label{thm: generalization of cover-zhang}
Let $r>1-1/d$. Let $f$ be a probability density function on $\mathbb{R}^d$. For any fixed integer $n\geq 2$, the identity 
$$
\sup_{X_i\sim f} h_r\left( \sum_{i=1}^n\lambda_iX_i\right) = h_r(X_1)
$$
holds for all $\lambda_i\geq0$ such that $\sum_{i=1}^n\lambda_i=1$
if and only if the density $f$ is $(r-1)$-concave.
\end{theorem}

The paper is organized as follows. In Section \ref{Renyi-CLT}, we explore the convergence along the CLT for $r$-R\'enyi entropies. For $r > 1$, the convergence is fully characterized for densities on $\R^d$, while for $r \in (0,1)$ sufficient conditions are obtained for a large class of densities. More precisely, we prove the convergence for log-concave densities and for compactly supported, spherically symmetric and unimodal densities. As an application, we prove in Section \ref{Renyi-EPI-sec} that a general $r$-R\'enyi EPI fails when $r \in (0,1)$, thus establishing Theorem \ref{thm: REPI failure}. We also complement this result by proving Theorems \ref{thm: s-concave REPI} and \ref{thm: alpha-modif}. In the last section, we provide an entropic characterization of the class of $s$-concave densities, and include a reverse R\'enyi EPI as an immediate consequence.

\section{Convergence along the CLT for R\'enyi entropies}\label{Renyi-CLT}

Let $\{X_n\}_{n\in\N}$ be a sequence of independent identically distributed (henceforth, i.i.d.) centered random vectors in $\R^d$ with finite covariance matrix. We denote by $Z_n$ the normalized sum
\begin{equation}\label{normal}
Z_n = \frac{X_1 + \cdots + X_n}{\sqrt{n}}.
\end{equation}
An important tool used to prove various forms of CLT is the characteristic function. Recall that the characteristic function of a random vector $X$ is defined by
$$
\varphi_X(t) = \EE\big[e^{i \langle t, X \rangle}\big], \quad t \in \R^d.
$$
Before providing sufficient conditions for the convergence along the CLT for R\'enyi entropy of order $r \in (0,1)$, we first extend \cite[Theorem 1.1]{BM18} to higher dimensions.


\begin{theorem}\label{r>1}

Let $r>1$. Let $X_1, \cdots, X_n$ be i.i.d. centered random vectors in $\R^d$. We denote by $\rho_n$ the density of $Z_n$ defined in \eqref{normal}. The following statements are equivalent.
\begin{enumerate}
\item $h_r ( Z_n ) \to h_r(Z)$ as $n \to + \infty$, where $Z$ is a Gaussian random vector with mean $0$ and the same covariance matrix as $X_1$.
\item $h_r(Z_{n_0})$ is finite for some integer $n_0$.
\item $\int_{\R^d} |\varphi_{X_1}(t)|^{\nu} \, dt < + \infty$ for some $\nu \geq 1$.
\item $Z_{n_0}$ has a bounded density $\rho_{n_0}$ for some integer $n_0$.
\end{enumerate}

\end{theorem}

\begin{proof}
$1 \Longrightarrow 2$:
Assume that $h_r \left( Z_n \right) \to h_r(Z)$ as $n \to + \infty$. Then there exists an integer $n_0$ such that
$$
h_r(Z) - 1 < h_r(Z_{n_0}) < h_r(Z) + 1.
$$
Since $h_r(Z)$ is finite, we conclude that $h_r(Z_{n_0})$ is finite as well.

$2 \Longrightarrow 3$: Assume that $h_r(Z_{n_0})$ is finite for some integer $n_0$. Then $Z_{n_0}$ has a density $\rho_{n_0}\in L^r(\R^d)$.

Case 1: If $r \geq 2$, we have $\rho_{n_0} \in L^2(\R^d)$. Using Plancherel's identity, we have $\varphi_{Z_{n_0}} \in L^2(\R^d)$. It follows that
$$
\int_{\R^d} |\varphi_{Z_{n_0}}(t)|^2 \, dt = \int_{\R^d} |\varphi_{X_1} \left( t/\sqrt{n_0} \right)|^{2n_0} \, dt < + \infty.
$$
For $\nu = 2 n_0$, we have
$$
\int_{\R^d} |\varphi_{X_1}(t)|^{\nu} \, dt < + \infty.
$$

Case 2: If $r \in (1,2)$, we apply the Hausdorff-Young inequality to obtain
$$
\|\varphi_{Z_{n_0}}\|_{L^{r'}} \leq \frac{1}{(2\pi)^{d/r'}} \|\rho_{n_0}\|_{L^r},
$$
where $r'$ is the conjugate of $r$ such that $1/r+1/r'=1$. Hence, for $\nu = r' n_0$, we have
$$
\int_{\R^d} |\varphi_{X_1}(t)|^{\nu} \, dt < + \infty.
$$

$3 \Longrightarrow 4$: Since $\int_{\R^d} |\varphi_{X_1}(t)|^{\nu} \, dt < + \infty$ for some $\nu \geq 1$, one may apply Gnedenko's local limit theorems (see \cite{GK68:book}), which is valid in arbitrary dimensions (see \cite{BRR76}). In particular, we have
\begin{eqnarray}\label{uniform}
\lim_{n \to + \infty} \sup_{x \in \R^d} |\rho_n(x) - \phi_{\Sigma}(x)| = 0,
\end{eqnarray}
where $\phi_{\Sigma}$ denotes the density of a Gaussian random vector with mean $0$ and the same covariance matrix as $X_1$. We deduce that there exists an integer $n_0$ and a constant $M>0$ such that $\rho_n \leq M$ for all $n \geq n_0$.

$4 \Longrightarrow 1$: Since $\rho_{n_0}$ is bounded, then $\rho_{n_0} \in L^2$, and we deduce by Plancherel's identity that $\int_{\R^d} |\varphi_{X_1}(t)|^{\nu} \, dt < + \infty$ for $\nu = 2n_0$. Hence, \eqref{uniform} holds and there exists $M>0$ such that $\rho_n \leq M$ for all $n \geq n_0$. Let us show that $\int_{\R^d} \rho_n(x)^rdx \to \int_{\R^d} \phi_{\Sigma}(x)^rdx$ as $n \to + \infty$, where $\phi_{\Sigma}$ denotes the density of a Gaussian random vector with mean $0$ and the same covariance matrix as $X_1$. By the CLT, for any $\eps>0$, there exists $T>0$ such that for all $n$ large enough,
$$
\int_{|x|>T} \rho_n(x)dx < \eps, 
$$
which implies that
$$
\int_{|x|>T} \rho_n(x)^rdx \leq M^{r-1} \int_{|x|>T} \rho_n(x)dx < M^{r-1} \eps.
$$
The function $\phi_{\Sigma}$ satisfies similar inequalities. Hence, for any $\delta > 0$, there exists $T > 0$ such that for all $n$ large enough,
$$
\left| \int_{|x|>T} \rho_n(x)^rdx - \int_{|x|>T} \phi_{\Sigma}(x)^rdx \right| < \delta.
$$
On the other hand, by \eqref{uniform}, for all $T>0$, the function $\rho_n^r(x) {\bf 1}_{\{|x| \leq T\}}$ converges everywhere to $\phi_{\Sigma}^r(x) {\bf 1}_{\{|x| \leq T\}}$  as $n \to  + \infty$. Since $\rho_n^r(x) {\bf 1}_{\{|x| \leq T\}}$ is dominated by the integrable function $M^r {\bf 1}_{\{|x| \leq T\}}$, one may use the Lebesgue dominated theorem to conclude that
$$
\lim_{n \to + \infty }\left| \int_{|x| \leq T} \rho_n(x)^rdx - \int_{|x| \leq T} \phi_{\Sigma}(x)^rdx \right| = 0.
$$
\end{proof}

\begin{remark}

Theorem \ref{r>1} fails for $r \in (0,1)$. For example, one can consider i.i.d. random vectors  with a bounded density $\rho(x)$  such that $\int_{\R^d} \rho(x)^rdx = + \infty$ (e.g., Cauchy-type distributions). The implication $4 \Longrightarrow 2$ (and thus $4 \Longrightarrow 1$) will not hold since by Jensen inequality $h_r(Z_n) \geq h_r(X_1 / \sqrt{n}) = \infty$ for all $n \geq 1$. As observed by Barron \cite{Bar86}, the implication $1 \Longrightarrow 4$ does not necessarily hold in the Shannon entropy case $r=1$.

\end{remark}



The following result yields a sufficient condition for convergence along the CLT to hold for R\'enyi entropies of order $r \in (0,1)$ for a large class of random vectors in $\R^d$. 

\begin{theorem}

Let $r \in (0,1)$. Let $X_1, \cdots, X_n$ be i.i.d. centered log-concave random vectors in $\R^d$. Then we have $h_r(Z_{n}) < +\infty$ for all $n \geq 1$, and
$$ \lim_{n\to\infty}h_r \left( Z_n \right)= h_r(Z),$$
where $Z_n$ is the normalized sum in \eqref{normal} and $Z$ is a Gaussian random vector with mean $0$ and the same covariance matrix as $X_1$.

\end{theorem}

\begin{proof}
Since log-concavity is preserved under independent sum, $Z_n$ is log-concave for all $n \geq 1$. Hence, for all $n \geq 1$, $Z_n$ has a bounded log-concave density $\rho_n$, which satisfies
$$
\rho_n(x) \leq e^{-a_n|x|+b_n},
$$
for all $x \in \R^d$, and for some constants $a_n > 0$, $b_n \in \R$ possibly depending on the dimension (see, e.g., \cite{BGVV-book-14}). Hence, for all $n \geq 1$, we have
$$
\int_{\R^d} \rho_n(x)^r \, dx \leq \int_{\R^d} e^{-r (a_n|x|+b_n)} \, dx < + \infty.
$$
We deduce that $h_r(Z_{n}) < +\infty$ for all $n \geq 1$.

The boundedness of $\rho_n$ implies that \eqref{uniform} holds, and thus there exists an integer $n_0$ such that for all $n \geq n_0$,
$$
\rho_n(0) > \frac{1}{2} \phi_{\Sigma}(0),
$$
where $\Sigma$ is the covariance matrix of $X_1$ (and thus does not depend on $n$). Moreover, since $\rho_n$ is log-concave, one has for all $x \in \R^d$ that
$$
\rho_n(rx) = \rho_n ((1-r)0+rx) \geq \rho_n(0)^{1-r} \rho_n(x)^r \geq \frac{1}{2^{1-r}} \phi_{\Sigma}(0)^{1-r} \rho_n(x)^r.
$$
Hence, for all $T>0$, we have
\begin{eqnarray*}
\int_{|x|>T} \rho_n(x)^r \, dx & \leq & \frac{2^{1-r}}{\phi_{\Sigma}(0)^{1-r}} \int_{|x|>T} \rho_n(r x) \, dx \\ & = & \frac{2^{1-r}}{r^d \phi_{\Sigma}(0)^{1-r}} \Prob{|Z_n| > rT} \\ & \leq & \frac{1}{T^2} \frac{2^{1-r} \EE[|X_1|^2]}{r^{d+2} \phi_{\Sigma}(0)^{1-r}},
\end{eqnarray*}
where the last inequality follows from Markov's inequality and the fact that
$$
\EE[|Z_n|^2] =\frac{\EE[|X_1|^2] + \cdots + \EE[|X_n|^2]}{n} = \EE[|X_1|^2].
$$
Hence, for every $\eps > 0$, one may choose a positive number $T$ such that for all $n$ large enough,
$$
\int_{|x|>T} \rho_n(x)^rdx < \eps, \qquad \int_{|x|>T} \phi_{\Sigma}(x)^rdx < \eps,
$$
and hence
$$
\left| \int_{|x| > T} \rho_n(x)^rdx - \int_{|x| > T} \phi_{\Sigma}(x)^rdx \right| < \eps.
$$
On the other hand, from \eqref{uniform}, we conclude as in the proof of Theorem \ref{r>1} that for all $T>0$,
$$
\lim_{n \to + \infty }\left| \int_{|x| \leq T} \rho_n(x)^rdx - \int_{|x| \leq T} \phi_{\Sigma}(x)^rdx \right| = 0.
$$
\end{proof}

A function $f \colon \R^d \to \R$ is called unimodal if the super-level sets $\{x\in\R^d: f(x)>t\}$ are convex for all $t\in\R$.  Next, we provide a convergence result for random vectors in $\R^d$ with unimodal densities under additional symmetry assumptions. First, we need the following stability result.

\begin{proposition}\label{stability}
The class of spherically symmetric and unimodal  random variables is stable under convolution.
\end{proposition}

\begin{proof}
Let $f_1$ and $f_2$ be two spherically symmetric and unimodal densities. By assumption, $f_i$ satisfy that $f_i(Tx) = f_i(x)$ for an orthogonal map $T$ and $|x| \leq |y|$ implies $f_i(x) \geq f_i(y)$.  By the layer cake decomposition, we write
$$
 f_i(x) = \int_0^\infty {\bf 1}_{\{(u,v): f_i(u) > v\}}(x, \lambda) d \lambda.
$$
Apply Fubini's theorem to obtain
\begin{eqnarray}
 f_1 \star f_2(x) &=& \int_{\mathbb{R}^d} f_1(x-y) f_2(y) dy \nonumber  \\
&=& \int_0^\infty \int_0^\infty \left( \int_{\mathbb{R}^d} {\bf 1}_{\{(u,v): f_1(u) > v\}}(x - y, \lambda_1) {\bf 1}_{\{(u,v): f_2(u) > v\}}(y, \lambda_2 )dy  
\right) d \lambda_1 d\lambda_2. \label{eq:integral we want in layer cake decomp}
\end{eqnarray}
Notice that by the spherical symmetry and decreasingness of $f_i$, 
the super-level set
$$
L_{\lambda_i} = \left\{ u : f_i(u) > \lambda_i \right\}
$$
is an origin symmetric ball.  Thus we can write the integrand in \eqref{eq:integral we want in layer cake decomp} as
$$
\int_{\mathbb{R}^d} {\bf 1}_{L_{\lambda_1}}(x - y) {\bf 1}_{L_{\lambda_2}}(y) dy = {\bf 1}_{L_{\lambda_1}} \star {\bf 1}_{L_{\lambda_2}}(x).
$$
This quantity is clearly dependent only on $|x|$, giving spherical symmetry.  In addition, as the convolution of two log-concave functions, 
${\bf 1}_{L_{\lambda_1}} \star {\bf 1}_{L_{\lambda_2}}$ is log-concave as well.  It follows that for every $\lambda_1, \lambda_2$, and $|x| \leq |y|$ we have 
$$
{\bf 1}_{L_{\lambda_1}} \star {\bf 1}_{L_{\lambda_2}}(x) \geq {\bf 1}_{L_{\lambda_1}} \star {\bf 1}_{L_{\lambda_2}}(y).
$$  
Integrating this inequality completes the proof.
\end{proof}

Let us establish large deviation and pointwise inequalities for compactly supported, spherically symmetric and unimodal densities.

\begin{theorem}[Hoeffding \cite{Hoe63}] \label{thm: Hoeffding}
Let $X_1, \cdots, X_n$ be independent random variables with mean 0 and bounded in $(a_i, b_i)$, respectively. One has for all $T>0$,
$$ 
\mathbb{P} \left( \sum_{i=1}^n X_i > T \right) \leq \exp \left(- \frac{ 2 T^2}{ \sum_{i=1}^n(b_i- a_i)^2 } \right).
$$
\end{theorem}

The following result is Hoeffding's inequality in higher dimensions.

\begin{lemma} \label{lem: Clumsy Mulitdimensional Hoeffding}
Let $X_1, \cdots, X_n$ be centered independent random vectors in $\R^d$ satisfying $\mathbb{P}( | X_i | > R) = 0$ for some $R>0$. One has for all $T>0$ that
$$
\mathbb{P} \left( \left|\frac{X_1 + \cdots + X_n}{\sqrt{n}} \right | > T \right) \leq 2 d \exp \left( - \frac{T^2}{2d^2 R^2} \right).
$$
\end{lemma}

\begin{proof}
Let $X_{i, j}$ be the $j$-th coordinate of the random vector $X_i$. Then we have
\begin{eqnarray}
\mathbb{P} \left( \left|\frac{X_1 + \cdots + X_n}{\sqrt{n}} \right | > T \right)
            &\leq& \label{eq: pigeon hole}
                \mathbb{P} \left( \bigcup_{j=1}^d \left\{|X_{1, j} + \cdots + X_{n, j}| > \frac{T \sqrt{n}}{d} \right\} \right)
                    \\
            &\leq& \label{eq: union bound}
                \sum_{j=1}^d \mathbb{P} \left( |X_{1, j} + \cdots + X_{n, j}| > \frac{T \sqrt{n}}{d} \right)
                    \\
            &\leq& \label{eq: application of Hoeff}
                2 d \exp \left( - \frac{T^2}{2d^2 R^2} \right),
    \end{eqnarray}
  where inequality \eqref{eq: pigeon hole} follows from the pigeon-hole principle, \eqref{eq: union bound} from a union bound, and \eqref{eq: application of Hoeff} follows from applying Theorem \ref{thm: Hoeffding} to $X_{1, j} + \cdots + X_{n, j}$ and $(-X_{1, j}) + \cdots + (-X_{n, j})$.
\end{proof}

We deduce the following pointwise estimate for unimodal spherically symmetric and bounded random variables.

\begin{corollary} \label{cor: Hoeffding decay away from zero}
Let $X_1, \cdots, X_n$ be i.i.d. random vectors with spherically symmetric unimodal density supported on the Euclidean ball $B_R=\{x: |x| \leq R\}$ for some $R >0$. Let $\rho_n$ denote the density of the normalized sum $Z_n$. Then there exists $c_d>0$ such that for all $n \geq 1$ and $|x| > 2$,
$$
\rho_n(x) \leq c_d \exp \left( - \frac{(|x| - 1)^2}{2d^2 R^2} \right).
$$
\end{corollary}

\begin{proof}
 Stating Lemma \ref{lem: Clumsy Mulitdimensional Hoeffding} in terms of $\rho_n$, we have
    \begin{eqnarray} \label{eq: Clumsy multidim hoeff in terms of density}
        \int_{|w| > T} \rho_n(w) dw 
            &\leq &
                2d \exp \left( - \frac{ T^2}{2d^2 R^2} \right).
    \end{eqnarray}
    Since the class of spherically symmetric unimodal random variables is stable under independent summation by Proposition \ref{stability}, $\rho_n$ is spherically symmetric and unimodal, so that
    \begin{eqnarray}
        \rho_n(x) 
            &\leq& \nonumber 
                \frac{\int_{B_{|x|} \backslash B_{|x|-1}} \rho_n(w) dw }{\Vol(B_{|x|} \backslash B_{|x|-1})  } 
                    \\
            &\leq& \label{eq: volume bound and tail thing}
                \frac{ \int_{|w| \geq |x| - 1} \rho_n(w) dw}{(2^d-1) \omega_d}
    \end{eqnarray}
    where $B_{|x|}$ represents the Euclidean ball of radius $|x|$ centered at the origin and $\omega_d$ is the volume of the unit ball.  Note that 
$$
\Vol(B_{|x|} \backslash B_{|x|-1}) = (|x|^d - (|x| - 1)^d) \omega_d \geq (2^d - 1) \omega_d,
$$
    since $t \mapsto t^d - (t-1)^d$ is increasing, so that \eqref{eq: volume bound and tail thing} follows.  Now applying \eqref{eq: Clumsy multidim hoeff in terms of density} we have
    \begin{eqnarray*}
        \rho_n(x) 
            &\leq&
                \frac{ \int_{|w| \geq |x| - 1} \rho_n(w) dw}{(2^d-1) \omega_d}
                    \\
            &\leq&
               \frac{ 2d}{(2^d -1)\omega_d} \exp \left( - \frac{(|x|-1)^2}{2d^2 R^2} \right)
    \end{eqnarray*}
    and our result holds with 
$$
c_d = \frac{ 2d}{(2^d -1)\omega_d}.
$$
\end{proof}

We are now ready to establish a convergence result for bounded spherically symmetric unimodal random vectors.

\begin{theorem}\label{CLT-radial-unim}
Let $r \in (0,1)$. Let $X_1, \cdots, X_n$ be i.i.d. random vectors in $\R^d$ with a spherically symmetric unimodal density with compact support. Then we have
$$
\lim_{n \to \infty} h_r (Z_n) = h_r(Z),
$$
where $Z_n$ is the normalized sum in \eqref{normal} and $Z$ is a Gaussian random vector with mean $0$ and the same covariance matrix as $X_1$.
\end{theorem}

\begin{proof}
Let us denote by $\rho_n$ the density of $Z_n$. Since $\rho_1$ is bounded, one may apply \eqref{uniform} together with Lebesgue dominated convergence to conclude that for all $T>0$,
$$
\lim_{n \to + \infty }\left| \int_{|x| \leq T} \rho_n(x)^rdx - \int_{|x| \leq T} \phi_{\Sigma}(x)^rdx \right| = 0.
$$
On the other hand, by Corollary \ref{cor: Hoeffding decay away from zero}, one may choose $T>0$ such that for all $n \geq 1$,
$$
\int_{|x|>T} \rho_n(x)^rdx < \eps, \qquad \int_{|x|>T} \phi_{\Sigma}(x)^rdx < \eps,
$$
and hence
$$
\left| \int_{|x| > T} \rho_n(x)^rdx - \int_{|x| > T} \phi_{\Sigma}(x)^rdx \right| < \eps.
$$
\end{proof}

\section{R\'enyi EPIs of order $r\in(0,1)$}\label{Renyi-EPI-sec}

A striking difference between R\'enyi EPIs of orders $r \in (0,1)$ and $r \geq 1$ is the lack of an absolute constant. Indeed, it was shown in \cite{BC15:1} that for $r \geq 1$ R\'enyi EPI of the form \eqref{eq:BC-REPI} holds for generic independent random vectors with an absolute constant $c \geq \frac{1}{e}r^{\frac 1 {r-1}}$. In the following subsection, we show that such a R\'enyi EPI does not hold for $r\in(0, 1)$. 

\subsection{Failure of a generic R\'enyi EPI}

\begin{definition}
For $r \in [0,\infty]$, we define $c_r$ as the largest number such that for all $n, d\geq 1$ and  any  independent random vectors $X_1, \cdots, X_n$ in $\R^d$, we have
\begin{equation}\label{eq:cr-def}
  N_r(X_1 +\cdots + X_n) \geq c_r \sum_{i=1}^n N_r(X_i).
\end{equation}

\end{definition}

Then we can rephrase Theorem \ref{thm: REPI failure} as follows.

\begin{theorem}\label{negative-REPI}
For $r \in (0,1)$, the constant $c_r$ defined in \eqref{eq:cr-def} satisfies $c_r = 0$.
\end{theorem}

The motivating observation for this line of argument is the fact that for $r \in (0,1)$, there exist distributions with finite covariance matrices and infinite $r$-R\'enyi entropies.  One might anticipate that this could contradict the existence of an  $r$-R\'enyi EPI, as the CLT forces the normalized sum of i.i.d. random vectors $X_1, \cdots, X_n$ drawn from such a distribution to become ``more Gaussian''.  Heuristically, one anticipates that $N_r(X_1+\cdots+X_n)/n = N_r(Z_n)$ should approach $N_r(Z)$ for large $n$, where $Z_n$ is the normalized sum in \eqref{normal} and $Z$ is a Gaussian vector with the same covariance matrix as $X_1$, while $\sum_{i=1}^n N_r(X_i) /n = N_r(X_1)$ is infinite.

\begin{proof}[Proof of Theorem \ref{negative-REPI}]
Let us consider the following density 
$$
 f_{R,p,d}(x) =  C_R (1+|x|)^{-p}{\bf 1}_{B_R}(x) \quad x \in \R^d,
$$
with $p,R >0$ and $C_R$ implicitly determined to make $ f_{R,p,d}$ a density. Since the density is spherically symmetric, its covariance matrix can be rewritten as $\sigma_R^2 I$ for some $\sigma_R>0$, where $I$ is the identity matrix.  Computing in spherical coordinates one can check that $\lim_{R \to \infty} C_R$ is finite for $p > d$, and we can thus define a density $f_{\infty,p,d}$. What is more, when $p > d+2$, the limiting density $f_{\infty, p,d}$ has a finite covariance matrix, and has finite R\'enyi entropy if and only if $p > d/r$.
    
For fixed $r \in (0,1)$, we take $p \in (d^*+2, d^*/r]$, where $d^*=\min\{d\in\N: d > 2r/(1-r) \}$ guarantees the existence of such $p$.  In this case, the limiting density $f_{\infty, p, d^*}$ is well defined and it has finite covariance matrix $\sigma_\infty^2 I$, but the corresponding $r$-R\'enyi entropy is infinite. Now we select independent random vectors $X_1, \cdots, X_n$ from the distribution $f_{R,p,d^*}$. Since $f_{R,p,d^*}$ is a spherically symmetric and unimodal density with compact support, we can apply Theorem \ref{CLT-radial-unim} to conclude that
$$
\lim_{n\to\infty} N_r(Z_n) = \sigma_R^2 N_r(Z_{Id}),
$$
where  $Z_n$ is the normalized sum in \eqref{normal} and  $Z_{Id}$ is the standard $d$-dimensional Gaussian. Since $\lim_{R \to \infty} \sigma_R=\sigma_\infty<\infty$, we can take $R$ large enough such that $|\sigma_R^2- \sigma_\infty^2|\leq 1$. Then we can take $n$ large enough such that
\begin{equation}\label{eq:upper bound zn}
N_r(Z_n) \leq (\sigma_\infty^2 + 2) N_r(Z_{Id}).
\end{equation}
Since the limiting density $f_{\infty, p, d^*}$ has infinite $r$-R\'enyi entropy, given $M >0$, we can take $R$ large enough such that 
\begin{equation}\label{eq:lower bound x1}
N_r(X_1) \geq M.
\end{equation}
Combining \eqref{eq:upper bound zn} and \eqref{eq:lower bound x1}, we conclude that for inequality \eqref{eq:cr-def} to hold we must have
$$
c_r\leq\frac{(\sigma_\infty^2 + 2) N_r(Z_{Id})}{M}
$$
for all $M >0$.  Then the statement follows from taking the limit $M \to \infty$.
\end{proof}

\begin{remark}

Random vectors in our proof has identical $s$-concave density with $s \leq -r/d$. In the following section, we provide a complementary result by showing that R\'enyi EPI of order $r \in (0,1)$ does hold for $s$-concave densities when $-r/d<s<0$.

\end{remark}

\subsection{R\'enyi EPIs for $s$-concave densities}

As showed above, a generic R\'enyi EPI of the form \eqref{eq:BC-REPI} fails for $r\in(0, 1)$. In this part, we establish R\'enyi EPIs of the forms \eqref{eq:BC-REPI} and \eqref{eq:alpha-modifi} for an important class of random vectors with $s$-concave densities (see \eqref{eq:s-concave}).

Following Lieb \cite{Lie78}, we prove Theorems \ref{thm: s-concave REPI} and \ref{thm: alpha-modif} by showing their equivalent linearizations. The following linearization of \eqref{eq:BC-REPI} and \eqref{eq:alpha-modifi} is due to Rioul \cite{rioul2018renyi}. The $c=1$ case was used in \cite{Li17}. 

\begin{theorem}[\cite{rioul2018renyi}]\label{lemma:linearization}
Let $X_1, \cdots, X_n$ be independent random vectors in $\R^d$. The following statements are equivalent.
\begin{enumerate}
\item There exist a constant $c>0$ and an exponent $\alpha>0$ such that
\begin{equation}\label{eq:exp}
N_r^{\alpha} \left( \sum_{i=1}^n X_i \right) \geq c \sum_{i=1}^n N_r^{\alpha} (X_i).
\end{equation}
\item For any $\lambda_1,\cdots, \lambda_n \geq 0$ such that $\sum_{i=1}^n\lambda_i=1$, one has
\begin{equation}\label{eq:linear}
h_r \left( \sum_{i=1}^n \sqrt{\lambda_i} X_i \right) - \sum_{i=1}^n \lambda_i h_r(X_i) \geq \frac{d}{2}\left( \frac{\log c}{\alpha} + \left( \frac{1}{\alpha} - 1 \right) H(\lambda) \right),
\end{equation}
where $H(\lambda) \triangleq H(\lambda_1,\cdots, \lambda_n)$ is the discrete entropy defined as
$$
H(\lambda) = - \sum_{i=1}^n \lambda_i \log\lambda_i.
$$
\end{enumerate}
\end{theorem}

Inequality \eqref{eq:linear} is the linearized form of inequality \eqref{eq:exp}. One of the ingredients used to establish \eqref{eq:linear} is Young's sharp convolution inequality \cite{Bec75,BL76b}. Its information-theoretic formulation was given in \cite{DCT91}, which we recall below. We denote by $r'$ the H\"{o}lder conjugate of $r$ such that $1/r+1/r'=1$.

\begin{theorem}[\cite{BL76b, DCT91}]\label{info-Young}
Let $r>0$. Let $\lambda_1, \cdots, \lambda_n \geq 0$ such that $\sum_{i=1}^n \lambda_i = 1$, and let $r_1, \cdots, r_n$ be positive reals such that $\lambda_i=r'/r_i'$. For any independent random vectors $X_1, \cdots, X_n$ in $\R^d$, one has
\begin{equation}\label{Young-info-theory}
 h_r \left( \sum_{i=1}^n \sqrt{\lambda_i} X_i \right) - \sum_{i=1}^n \lambda_i h_{r_i}(X_i) \geq \frac{d}{2} r' \left( \frac{\log r}{r} - \sum_{i=1}^n \frac{\log r_i}{r_i} \right).
\end{equation}

\end{theorem}

The second ingredient is a comparison between R\'enyi entropies $h_r$ and $h_{r_i}$. When $r>1$, we have $1<r_i<r$, and Jensen's inequality implies that $h_r\leq h_{r_i}$. In this case, one can deduce \eqref{eq:linear} from \eqref{Young-info-theory} with $h_{r_i}$ replaced by $h_r$. However, when $r \in (0,1)$, the order of $r$ and $r_i$ are reversed, i.e., $0<r<r_i<1$, and we need a reverse entropy comparison inequality. The so-called $s$-concave densities do satisfy such a reverse entropy comparison inequality. The following result of Fradelizi, Li, and Madiman \cite{FLM15} serves this purpose.

\begin{theorem}[\cite{FLM15}]\label{main}

Let $s \in \R$. Let $f \colon \R^d \to [0,+\infty)$ be an integrable $s$-concave function. The function
$$
G(r) = C(r) \int_{\R^d} f(x)^r \, dx
$$
is log-concave for $r > \max\{0, -sd\}$, where
\begin{equation}\label{eq:c-r}
C(r) = (r+s) \cdots (r+sd).
\end{equation}

\end{theorem}

We deduce the following R\'enyi entropic comparison for random vectors with  $s$-concave densities.

\begin{corollary}\label{rev-holder}

Let $X$ be a random vector in $\R^d$ with a $s$-concave density. For $-sd < r < q<1$, we have
$$
h_{q}(X) \geq h_r(X) + \log  \frac{C(r)^{\frac{1}{1-r}} C(1)^{\frac{q-r}{(1-q)(1-r)}}}{C(q)^{\frac{1}{1-q}}}.
$$

\end{corollary}

\begin{proof}
Write $q=(1-\lambda)\cdot r + \lambda \cdot 1$. Using the log-concavity of the function $G$ in Theorem \ref{main}, we have
$$
G(q) \geq G(r)^{1-\lambda} G(1)^{\lambda} = G(r)^{\frac{1-q}{1-r}} G(1)^{\frac{q-r}{1-r}}.
$$
The above inequality can be rewritten in terms of entropy power as follows
$$
C(q)^{\frac{2}{d} \cdot\frac{1}{1-q}} N_q(X) \geq C(r)^{\frac{2}{d}\cdot\frac{1-q}{1-r}\cdot \frac{1}{1-q}} N_r(X) C(1)^{\frac{2}{d}\cdot\frac{q-r}{1-r}\cdot\frac{1}{1-q}}.
$$
The desired statement follows from taking the logarithm of both sides of the above inequality.
\end{proof}

Theorem \ref{info-Young} together with Corollary \ref{rev-holder} yields the following R\'enyi EPI with a single R\'enyi parameter $r \in (0,1)$ for  $s$-concave densities.

\begin{theorem}\label{lemma:young-s-concave}
Let $s\in(-1/d, 0)$ and $r \in (-sd, 1)$. Let $X_1, \cdots, X_n$ be independent random vectors in $\R^d$ with $s$-concave densities. For all $\lambda = (\lambda_1, \cdots, \lambda_n) \in [0,1]^n$ such that $\sum_{i=1}^n \lambda_i=1$, we have
$$
h_r \left( \sum_{i=1}^n \sqrt{\lambda_i} X_i \right) - \sum_{i=1}^n \lambda_i h_r(X_i)\geq \frac{d}{2}A(\lambda)+\sum_{k=1}^d g_k(\lambda),
$$
where
\begin{equation*}\label{eq:a-lambda}
\begin{aligned}
A(\lambda) & =  r'\left(\left(1-\frac{1}{r'}\right)\log\left(1-\frac{1}{r'}\right)-\sum_{i=1}^n\left(1-\frac{\lambda_i}{r'}\right)\log\left(1-\frac{\lambda_i}{r'}\right)\right), \\
g_k(\lambda) & =  (1-n)r'\log(1+ks)+(1-r')\log\Big(1+\frac{ks}{r}\Big)+r'\sum_{i=1}^n\left(1-\frac{\lambda_i}{r'}\right)\log\left(1+ks\left(1-\frac{\lambda_i}{r'}\right)\right).
\end{aligned}
\end{equation*}
\end{theorem}

\begin{proof}
Let $r_i$ be defined by $\lambda_i=r'/r_i'$, where $r'$ and $r_i'$ are H\"{o}lder conjugates of $r$ and $r_i$, respectively. Combining Theorem \ref{info-Young} with Corollary \ref{rev-holder}, we have
\begin{equation}\label{eq:entropy-linear-ineq}
h_r \left( \sum_{i=1}^n \sqrt{\lambda_i} X_i \right) - \sum_{i=1}^n \lambda_i h_r(X_i) \geq \frac{d}{2} r' \left( \frac{\log r}{r} - \sum_{i=1}^n \frac{\log r_i}{r_i} \right) + \sum_{i=1}^n \lambda_i \log \frac{C(r)^{\frac{1}{1-r}}C(1)^{\frac{r_i-r}{(1-r_i)(1-r)}}}{C(r_i)^{\frac{1}{1-r_i}}}.
\end{equation}
Notice that $C(r)=r^dD(r)$, where $C(r)$ is given in \eqref{eq:c-r} and $D(r)=(1+s/r)\cdots(1+sd/r)$. Thus,
\begin{eqnarray}\label{eq:2nd term}
& ~ & \sum_{i=1}^n \lambda_i \log \frac{C(r)^{\frac{1}{1-r}}C(1)^{\frac{r_i-r}{(1-r_i)(1-r)}}}{C(r_i)^{\frac{1}{1-r_i}}} \nonumber \\  
& = & \sum_{i=1}^n\lambda_i\left(\frac{\log D(r)}{1-r}\!+\!\left(\frac{1}{1-r_i}\!-\!\frac{1}{1-r}\right)\log D(1)\!-\!\frac{\log D(r_i)}{1-r_i}\right) \!+\! d\left(\frac{\log r}{1-r}\!-\!\sum_{i=1}^n\lambda_i\frac{\log r_i}{1-r_i}\right).
\end{eqnarray}
Using the identities $1/(1-r)=1-r'$ and $\lambda_i/(1-r_i)=\lambda_i-r'$, we have
\begin{eqnarray}\label{eq:1st term}
& ~ & \sum_{i=1}^n\lambda_i\left(\frac{\log D(r)}{1-r}+\left(\frac{1}{1-r_i}-\frac{1}{1-r}\right)\log D(1)-\frac{\log D(r_i)}{1-r_i}\right) \nonumber \\ 
& = & (1-r')\log D(r)+(1-n)r'\log D(1)+\sum_{k=1}^d\sum_{i=1}^n(r'-\lambda_i)\log\left(1+\frac{ks}{r_i}\right) \nonumber\\
&=&\sum_{k=1}^d\left((1-r')\log\left(1+\frac{ks}{r}\right)+(1-n)r'\log(1+ks)+\sum_{i=1}^n(r'-\lambda_i)\log\left(1+\frac{ks}{r_i}\right)
\right) \nonumber\\
&=& \sum_{k=1}^d g_k(\lambda).
\end{eqnarray}
The last identity follows from $1/r_i=1-\lambda_i/r'$. Using \eqref{eq:1st term} and \eqref{eq:2nd term}, the RHS of \eqref{eq:entropy-linear-ineq} can be written as
$$
\frac{d}{2} r' \left( \frac{\log r}{r} - \sum_{i=1}^n \frac{\log r_i}{r_i} \right) + d\left(\frac{\log r}{1-r}-\sum_{i=1}^n\lambda_i\frac{\log r_i}{1-r_i}\right) + \sum_{k=1}^d g_k(\lambda) = \frac{d}{2}A(\lambda) + \sum_{k=1}^d g_k(\lambda).
$$
This concludes the proof.
\end{proof}

Having Theorems \ref{lemma:linearization} and \ref{lemma:young-s-concave} at hand, we are ready to prove Theorems 2 and 3.

\subsubsection{Proof of Theorem \ref{thm: s-concave REPI}}





Put Theorems \ref{lemma:linearization} and \ref{lemma:young-s-concave} together. Then it suffices to find $c$ such that the following inequality 
$$
\frac{d}{2}A(\lambda)+\sum_{k=1}^d g_k(\lambda)\geq\frac{d}{2}\log c
$$
holds for all $\lambda=(\lambda_1, \cdots, \lambda_n) \in [0,1]^n$ such that $\sum_{i=1}^n\lambda_i = 1$. Hence, we can set 
$$
c=\inf_\lambda\exp\left(A(\lambda)+\frac{2}{d}\sum_{k=1}^d g_k(\lambda)\right),
$$
where the infimum runs over all $\lambda=(\lambda_1, \cdots, \lambda_n) \in [0,1]^n$ such that $\sum_{i=1}^n\lambda_i = 1$.
For fixed $r$, both $A(\lambda)$ and $g_k(\lambda)$ are sum of one-dimensional convex functions of the form $(1+x)\log(1+x)$. Furthermore, both $A(\lambda)$ and $g_k(\lambda)$ are permutation invariant. Hence, the minimum is achieved at $\lambda=(1/n, \cdots, 1/n)$. This yields the numerical  value of $c$ in Theorem \ref{thm: s-concave REPI}.

\subsubsection{Proof of Theorem \ref{thm: alpha-modif}}

The following lemma in \cite{MM18-IEEE} serves us in the proof of Theorem \ref{thm: alpha-modif}.

\begin{lemma}[\cite{MM18-IEEE}]\label{lem:calc}
Let $c>0$. Let $L,F : [0, c] \to [0,\infty)$ be twice differentiable on $(0,c]$, continuous on $[0,c]$, such that $L(0) = F(0)= 0$ and $L'(c)=F'(c) =0$.  Let us also assume that $F(x)>0$ for $x>0$, that $F$ is strictly increasing, and that $F'$ is strictly decreasing. Then $\frac{L''}{F''}$ increasing on $(0,c)$ implies that $\frac{L}{F}$ is increasing on $(0,c)$ as well. In particular,
$$
\max_{x \in [0,c]} \frac{L(x)}{F(x)} = \frac{L(c)}{F(c)}.
$$
\end{lemma}

\begin{proof}[Proof of Theorem \ref{thm: alpha-modif}]
Apply Theorems \ref{lemma:linearization} and  \ref{lemma:young-s-concave} with $n=2$. Then it suffices to find $\alpha$ such that for all $\lambda \in [0,1]$ we have
$$
\frac{d}{2}A(\lambda)+\sum_{k=1}^d g_k(\lambda)\geq\frac{d}{2}\left(\frac{1}{\alpha}-1\right)H(\lambda),
$$
where
\begin{eqnarray*}
A(\lambda) & = & r'\left(\left(1\!-\!\frac{1}{r'}\right)\log\left(1\!-\!\frac{1}{r'}\right)\!-\!\left(1\!-\!\frac{\lambda}{r'}\right)\log\left(1\!-\!\frac{\lambda}{r'}\right)\!-\!\left(1\!-\!\frac{1\!-\!\lambda}{r'}\right)\log\left(1\!-\!\frac{1\!-\!\lambda}{r'}\right)\right), \\
g_k(\lambda)&=&(1-r')\log\Big(1+\frac{ks}{r}\Big)-r'\log(1+ks)  \\
&+&r'\left(\left(1\!-\!\frac{\lambda}{r'}\right)\log\left(1\!+\!ks\left(1\!-\!\frac{\lambda}{r'}\right)\right)\!+\!\left(1\!-\!\frac{1\!-\!\lambda}{r'}\right)\log\left(1\!+\!ks\left(1\!-\!\frac{1\!-\!\lambda}{r'}\right)\right)\right).
\end{eqnarray*}
We can set
\begin{equation}\label{alpha-val}
\alpha=\left(1-\sup_{0\leq \lambda\leq 1}\left(-\frac{A(\lambda)}{H(\lambda)}-\frac{2}{d}\sum_{k=1}^d\frac{g_k(\lambda)}{H(\lambda)}\right)\right)^{-1}.
\end{equation}
We will show that the optimal value is achieved at $\lambda=1/2$. Since the function is symmetric about $\lambda=1/2$, it suffices to show that
\begin{equation}\label{eq:obj-opt}
-\frac{A(\lambda)}{H(\lambda)}-\frac{2}{d}\sum_{k=1}^n\frac{g_k(\lambda)}{H(\lambda)}
\end{equation}
is increasing on $[0, 1/2]$. It has been shown in \cite{Li17} that $-A(\lambda)/H(\lambda)$ is increasing on $[0, 1/2]$. We will show that for each $k=1,\cdots, n$ the function $-g_k(\lambda)/H(\lambda)$ is also increasing on $[0, 1/2]$. One can check that $-g_k(\lambda)$ and $H(\lambda)$ satisfy the conditions in Lemma \ref{lem:calc}. 
Hence, it suffices to show that $-g_k''(\lambda)/H''(\lambda)$ is increasing on $[0, 1/2]$. 
Elementary calculation yields that
$$
H''(\lambda)=-\frac{1}{\lambda(1-\lambda)}.
$$
Define $x=\frac{\lambda}{|r'|}$ and $y=\frac{1-\lambda}{|r'|}=\frac{1}{|r'|} - x$. Then one can check that
$$
-g_k''(\lambda)=\frac{ks}{|r'|}\left( \frac{1}{1+ks(1+x)} + \frac{1}{1+ks(1+y)} + \frac{1}{(1+ks(1+x))^2} + \frac{1}{(1+ks(1+y))^2} \right).
$$
Hence, we have
$$
-\frac{g_k''(\lambda)}{H''(\lambda)} = ksr'W(x), 
$$
where
$$
W(x)= xy\left( \frac{1}{1+ks(1+x)} + \frac{1}{1+ks(1+y)} + \frac{1}{(1+ks(1+x))^2} + \frac{1}{(1+ks(1+y))^2} \right). 
$$
Since $s, r'<0$, it suffices to show that $W(x)$ is increasing on $[0, \frac{1}{2|r'|}]$. We rewrite $W$ as follows
$$
W(x) = W_1(x) + W_2(x), 
$$
where
\begin{eqnarray}
W_1(x) &=& xy\left(\frac{1}{1+ks(1+x)} + \frac{1}{1+ks(1+y)}\right)\notag,\\
W_2(x)&=& xy\left(\frac{1}{(1+ks(1+x))^2} + \frac{1}{(1+ks(1+y))^2} \right) \label{eq:w2}. 
\end{eqnarray}
We will show that both $W_1(x)$ and $W_2(x)$ are increasing on $[0, \frac{1}{2|r'|}]$.\\

Now let us focus on $W_1$. Since $y = \frac{1}{|r'|} - x$, one can check that
$$
W_1'(x)=\left( \frac{1}{|r'|} \!-\! 2x \right) \left( \frac{1}{1\!+\!ks(1\!+\!x)} \!+\! \frac{1}{1\!+\!ks(1\!+\!y)} \right) - ks xy \left( \frac{1}{(1\!+\!ks(1\!+\!x))^2} \!-\! \frac{1}{(1\!+\!ks(1\!+\!y))^2} \right).
$$
Let us denote
\begin{eqnarray}
a &\triangleq& a(x) = 1+ks(1+x) \label{eq:a},\\
b &\triangleq& b(x) = 1+ks(1+y) = 1+ks\left(\frac{1}{|r'|} - x + 1\right) \label{eq:b}.
\end{eqnarray}
The condition $r > -sd $ implies that $ a, b\geq 0$.
With these notations, we have
\begin{eqnarray*}
W_1'(x) 
& = & \left( \frac{1}{a} + \frac{1}{b} \right) \left( \frac{1}{|r'|} - 2x -ks x y \left( \frac{1}{a} - \frac{1}{b} \right) \right)\\
&=&\left( \frac{1}{a} + \frac{1}{b} \right) \left( \frac{1}{|r'|} - 2x \right) \left( 1 - (ks)^2  \frac{xy}{ab} \right).
\end{eqnarray*}
The last identity follows from
$$
\frac{1}{a} - \frac{1}{b} = \frac{ks}{ab} \left( \frac{1}{|r'|} - 2x \right).
$$
Since $a, b\geq 0$ and $x \in [0, \frac{1}{2|r'|}]$, it suffices to show that
$$
ab - (ks)^2 xy \geq 0.
$$
Using \eqref{eq:a} and \eqref{eq:b}, we have
$$
ab - (ks)^2 xy=(1+ks)\left(1+\frac{ks}{r}\right).
$$
Then the desired statement follows from that $s>-1/d$ and $r >-sd$. We conclude that $W_1$ is increasing on $[0, \frac{1}{2|r'|}]$.

It remains to show that $W_2(x)$ is increasing on $[0, \frac{1}{2|r'|}]$. Recall the definition of $W_2(x)$ in \eqref{eq:w2}, one can check that
\begin{eqnarray*} 
W_2'(x) & = & \left( \frac{1}{|r'|} - 2x \right) \left( \frac{1}{a^2} + \frac{1}{b^2} \right) - 2ks x y \left( \frac{1}{a^3} - \frac{1}{b^3} \right) \\ 
& = & \frac{b-a}{ks} \left( \frac{1}{a^2} + \frac{1}{b^2} \right) - 2ks x y \left( \frac{1}{a^3} - \frac{1}{b^3} \right) \\ 
& = & \frac{b-a}{ksa^3 b^3}T(x),
\end{eqnarray*}
where $a$ and $b$ are defined in \eqref{eq:a} and \eqref{eq:b}, and
$$ 
T(x)= ab(a^2 + b^2) - 2 k^2s^2 x y (a^2 + ab + b^2). 
$$
Since
$$
\frac{b-a}{ks} = \frac{1}{|r'|} - 2x\geq0,~~x\in[0, \frac{1}{2|r'|}], 
$$
it suffices to show that $T(x)\geq0$ for $[0, \frac{1}{2|r'|}]$. Using the identity
$$
a'(x)b(x) + a(x)b'(x) = ks(b-a) = - a(x)a'(x) - b(x)b'(x),
$$
one can check that
$$
T'(x) = ks(a-b)U(x),
$$
where
$$
U(x) = a^2 + b^2 + 4ab - 2 k^2s^2 x y.
$$
Notice that $U'(x) \equiv 0$, which implies that $U(x)$ is a constant. Since $a, b\geq0$, we have
$$
U(0) = a^2 + b^2 + 4ab > 0.
$$
Hence, $T'(x) \leq 0$, i.e., $T(x)$ is decreasing. Therefore, since $a=b$ when $x=\frac{1}{2|r'|}$, we have
$$
T(x) \geq T \left( \frac{1}{2 |r'|} \right)=2a^2(a^2-3k^2s^2x^2) \quad \text{at} \, x = \frac{1}{2 |r'|}.
$$
It suffices to have
$$
a^2 \geq 3 k^2s^2 x^2, \quad x = \frac{1}{2 |r'|},
$$
which is equivalent to
$$
\frac{1}{|r'|} \leq \frac{2}{1+\sqrt{3}} \left( \frac{1}{k|s|} - 1 \right).
$$
This finishes the proof that every $-g_k(\lambda)/H(\lambda)$ is also increasing on $[0,1/2]$. Then the numerical value of $\alpha$ in Theorem \ref{thm: alpha-modif} follows from setting $\lambda = 1/2$ in \eqref{alpha-val}.
\end{proof}

\begin{remark}
Our optimization argument heavily relies on the fact that $-A(\lambda)/H(\lambda)$ and $-g_k(\lambda)/H(\lambda)$ are monotonically increasing for $\lambda\in[0, 1/2]$. As observed in \cite{Li17}, the monotonicity of $-A(\lambda)/H(\lambda)$ does not depend on the value of $r$. Numerical examples show that $-g_k(\lambda)/H(\lambda)$, even the whole quantity in \eqref{eq:obj-opt}, is
not monotone when $r$ is small. This is one of the reasons for the restriction $r>r_0$.
\end{remark}

\begin{remark}
Note that the condition $r > -sd$ of Theorem \ref{main} can be rewritten as
$$
\frac{1}{|r'|} < \left( \frac{1}{d|s|} - 1 \right).
$$
We do not know whether Theorem \ref{thm: alpha-modif} holds when
$$
\frac{2}{1+\sqrt{3}} \left( \frac{1}{d|s|} - 1 \right) < \frac{1}{|r'|} < \left( \frac{1}{d|s|} - 1 \right).
$$
\end{remark}

\section{An entropic characterization of $s$-concave densities}\label{sec:max-ent}

Let $X$ and $Y$ be real-valued random variables (possibly dependent) with the identical density $f$. Cover and Zhang \cite{CZ94} proved that
$$
h(X+Y)\leq h(2X)
$$
holds for every coupling of $X$ and $Y$ if and only if $f$ is log-concave. This yields an entropic characterization of one-dimensional log-concave densities. We will extend Cover and Zhang's result to R\'enyi entropies of random vectors with $s$-concave densities (defined in \eqref{eq:s-concave}), which particularly include log-concave densities as a special case. This was previously proved in \cite{LM18:isit} when $f$ is continuous.

Firstly, we introduce some classical variations of convexity and concavity which will be needed in our proof.

\begin{definition}\label{def:lambda-convex}
Let $\lambda \in (0,1)$ be fixed. A function $f \colon \R^d \to \R$ with convex support is called almost $\lambda$-convex if the following inequality
\begin{eqnarray}\label{eq:convex-1}
f((1-\lambda)x + \lambda y) \leq (1-\lambda) f(x) + \lambda f(y)
\end{eqnarray}
holds for almost every pair $x, y$ in the domain of $f$. We say that $f$ is $\lambda$-convex if the above inequality holds for every pair $x, y$ in the domain of $f$. Particularly, for $\lambda=1/2$, it is usually called mid-convex or Jensen convex. We say that $f$ is convex if $f$ is $\lambda$-convex for any $\lambda \in (0,1)$.
\end{definition}

One can define almost $\lambda$-concavity, $\lambda$-concavity and concavity by reversing inequality \eqref{eq:convex-1}. Adamek \cite[Theorem 1]{adamek2003almost} showed that an almost $\lambda$-convex function is identical to a $\lambda$-convex function except on a set of Lebesgue measure 0. (To apply the theorem there, one can take the ideals $\mathcal{I}_1$ and $\mathcal{I}_2$ as the family of sets with Lebesgue measure 0 in $\R^d$ and $\R^{2d}$, respectively). In general, $\lambda$-convexity is not equivalent to convexity, as it is not a strong enough notion to imply continuity, at least not in a logical framework that accepts the axiom of choice.  Indeed, counterexamples can be constructed using a Hamel basis for $\R$ as a vector space over $\mathbb{Q}$.  However, in the case that $f$ is Lebesgue measurable, a classical result of Blumberg \cite{Blu19} and Sierpinski \cite{Sie20} (see also \cite{CM} in more general setting) shows that $\lambda$-convexity implies continuity, and thus convexity.

\begin{theorem}\label{thm:max-entropy}
Let  $s>-1/d$ and we define $r=1+s$.  Let $f$ be a probability density on $\mathbb{R}^d$. The following statements are equivalent.
\begin{enumerate}
\item The density $f$ is $s$-concave.
\item For any $\lambda\in(0, 1)$, we have $h_r(\lambda X+(1-\lambda)Y)\leq h_r(X)$ for any random vectors $X$ and $Y$ with the identical density $f$.
\item We have $h_r\left(\frac{X+Y}{2}\right)\leq h_r(X)$ for any random vectors $X$ and $Y$ with the identical density $f$.
\end{enumerate}

\end{theorem}

\begin{proof}
We only prove the statement for $s>0$, or equivalently $r>1$. The proof for $-1/d<s<0$, or equivalently $1-1/d<r<1$, is similar and sketched below.

$1 \Longrightarrow 2$: The proof is taken from \cite{LM18:isit}. We include it for completeness.  Let $g$ be the density of $\lambda X+(1-\lambda)Y$. Then we have  
\begin{eqnarray}
h_r(X) 
    &=& 
        \frac{1}{1-r}\log \mathbb{E}f^{r-1}(X) \notag
            \\
    &=& 
        \frac{1}{1-r}\log (\lambda \mathbb{E} f^{r-1}(X)+(1-\lambda)\mathbb{E} f^{r-1}(Y)) \label{eq:identical distri}
            \\
    &\geq& 
        \frac{1}{1-r}\log \mathbb{E} f^{r-1}(\lambda X+(1-\lambda)Y) \label{eq:concavity}
            \\
    &=& 
            \frac{1}{1-r}\log\int_{\R^d} f(x)^{r-1}g(x)dx\notag
                \\
    &\geq& 
        \frac{1}{1-r}\log \left(\int_{\R^d} f(x)^{r}dx\right)^{1-\frac{1}{r}}\left(\int_{\R^d} g(x)^{r}dx\right)^{\frac{1}{r}} \label{eq:holder}
            \\
    &=&
        \frac{r-1}{r}h_r(X)+\frac{1}{r}h_r(\lambda X+(1-\lambda)Y)\notag.
\end{eqnarray}
This is equivalent to the desired statement. Identity \eqref{eq:identical distri} follows from the assumption that $X$ and $Y$ have the same distribution. In inequality \eqref{eq:concavity}, we use the concavity of $f^{r-1}$ and the fact that $\frac{1}{1-r}\log x$ is decreasing when $r>1$. Inequality \eqref{eq:holder} follows from H\"{o}lder's inequality and the fact that $\frac{1}{1-r}\log x$ is decreasing when $r>1$. 
For $1-1/d<r<1$, the statement follows from the same argument in conjunction with the convexity of $f^{r-1}$, the converse of H\"{o}lder's inequality and the fact that $\frac{1}{1-r}\log x$ is increasing when $0<r<1$.

$2 \Longrightarrow 3$: Obvious by taking $\lambda=\frac{1}{2}$. 

$3 \Longrightarrow 1$: We will prove the statement by contradiction. We first show an example borrowed from Cover and Zhang \cite{CZ94} to illustrate the ``mass transferring'' argument used in our proof. Consider the density $f(x)=3/2$ in the intervals $(0, 1/3)$ and $(2/3, 1)$. It is clear that $f$ is not $(r-1)$-concave. The joint distribution of $(X, Y)$ with $Y\equiv X$ is supported on the diagonal line $y=x$. The Radon-Nikodym derivative $g$ with respect to the one-dimensional Lebesgue measure on the line $y=x$ exists and is shown in Figure 1. We remove some ``mass'' from the diagonal line $y=x$ to the lines $y=x-2/3$ and $y=x+2/3$. The new Radon-Nikodym derivative $\hat{g}$ is shown in Figure 2. Let $(\hat{X}, \hat{Y})$ be a pair of random variables whose joint distribution possesses this new Radon-Nikodym derivative. It is easy to see that $\hat{X}$ and $\hat{Y}$ still have the same density $f$. But $\hat{X}+\hat{Y}$ is uniformly distributed on $(0, 2)$, and thus $h_r(\hat{X}+\hat{Y})=\log 2$. One can check that $h_r(2X)=\log(4/3)$.

\begin{figure}
    \centering
    \begin{minipage}{0.45\textwidth}
        \centering
        \begin{tikzpicture}[scale=2,extended line/.style={shorten >=-#1,shorten <=-#1},]
\draw [->](0,-0)--(0,2) node[right]{$y$};
\draw [->](0,0)--(2,0) node[right]{$x$};
\draw (0,0)--(0.5,0.5)node[anchor=south west]{$g=\frac{3}{2\sqrt 2}$};
\draw (1,1)--(1.5,1.5)node[anchor=south west]{$g=\frac{3}{2\sqrt 2}$};
\foreach \x/\xtext in {0.5/{1/3}, 1/{2/3}, 1.5/1}
{\draw (\x cm,1pt ) -- (\x cm,-1pt ) node[anchor=north] {$\xtext$};}
\foreach \y/\ytext in {0.5/{1/3}, 1/{2/3}, 1.5/1}
{\draw (1pt,\y cm) -- (-1pt ,\y cm) node[anchor=east] {$\ytext$};}
\end{tikzpicture} 
        \caption{$g$}
    \end{minipage}\hfill
    \begin{minipage}{0.45\textwidth}
        \centering
        \begin{tikzpicture}[scale=2,extended line/.style={shorten >=-#1,shorten <=-#1},]
\draw [->](0,-0)--(0,2) node[right]{$y$};
\draw [->](0,0)--(2,0) node[right]{$x$};
\draw (0,0)--(0.5,0.5)node[anchor=south west]{$\hat{g}=\frac{1}{\sqrt 2}$};
\draw (1,1)--(1.5,1.5)node[anchor=south west]{$\hat{g}=\frac{1}{\sqrt 2}$};
\draw (0,1)--(0.5,1.5)node[anchor=south west]{$\hat{g}=\frac{1}{2\sqrt 2}$};
\draw (1,0)--(1.5,0.5)node[anchor=south west]{$\hat{g}=\frac{1}{2\sqrt 2}$};
\foreach \x/\xtext in {0.5/{1/3}, 1/{2/3}, 1.5/1}
{\draw (\x cm,1pt ) -- (\x cm,-1pt ) node[anchor=north] {$\xtext$};}
\foreach \y/\ytext in {0.5/{1/3}, 1/{2/3}, 1.5/1}
{\draw (1pt,\y cm) -- (-1pt ,\y cm) node[anchor=east] {$\ytext$};}
\end{tikzpicture}
        \caption{$\hat{g}$}
    \end{minipage}
\end{figure}

Now we turn to the general case. Suppose that $f$ is not $(r-1)$-concave, i.e., $f^{r-1}$ is not concave (for $r>1$). We claim that there exists a set $A \subseteq \R^{2d}$ of positive Lebesgue measure on $\R^{2d}$ such that the inequality \begin{equation}\label{eq:not-concave}
2 f^{r-1}\left(\frac{x+y}{2}\right)<f^{r-1}(x)+f^{r-1}(y)
\end{equation} holds for all $(x,y) \in A$.
 Otherwise, the converse of \eqref{eq:not-concave} holds for almost every pair $(x, y)$, and thus $f^{r-1}$ is an almost mid-concave function (i.e., 1/2-concave). By Theorem 1 in  \cite{adamek2003almost}, $f^{r-1}$ is identical to a mid-concave function except on a set of Lebesgue measure 0. Without changing the distribution, we can modify $f$ such that $f^{r-1}$ is mid-concave. Using the equivalence of mid-concavity and concavity (under the Lebesgue measurability), after modification, $f^{r-1}$ is concave, i.e., $f$ is $(r-1)$-concave. This contradicts our assumption. Hence, there exists such a set $A$ with positive Lebesgue measure on $\R^{2d}$. Then there exists $y$ such that \eqref{eq:not-concave} holds for a set of $x$ with positive Lebesgue measure on $\R^d$.
 We rephrase this statement in a
form suitable for our purpose. There is $x_0\neq0$ such that the set 
\begin{equation}\label{eq:lambda}
\Lambda =\big\{x\in \R^d: 2f(x)^{r-1}<f(x+x_0)^{r-1}+f(x-x_0)^{r-1}
\big\}
\end{equation}
has positive Lebesgue measure on $\R^d$. For $\epsilon>0$,  we denote by $\Lambda(\epsilon)$ a ball of radius $\epsilon$ whose intersection with $\Lambda$ has positive Lebesgue measure on $\R^d$. Consider $(X, Y)$ such that $X\equiv Y$, where $X$ and $Y$ have the identical density $f$. Let $g(x, y)$ be the Radon-Nikodym derivative of $(X, Y)$ with respect to the $d$-dimensional Lebesgue measure on the ``diagonal line'' $y=x$. Now we build a new density $\hat{g}$ by translating a small amount of ``mass'' from ``diagonal points'' $(x-x_0, x-x_0)$ and $(x+x_0, x+x_0)$ to ``off-diagonal points'' $(x-x_0, x+x_0)$ and $(x+x_0, x-x_0)$. 
 To be more precise, we define the new joint density $\hat{g}$ as
\begin{eqnarray*}
\hat{g}(x, y)=g(x, y) {\bf 1}_{\{x=y\}}-\sqrt{d/2}\delta({\bf 1}_{\{(x-x_0, x-x_0): x\in \Lambda(\epsilon)\}}+{\bf 1}_{\{(x+x_0, x+x_0): x\in \Lambda(\epsilon)\}})\\
+\sqrt{d/2}\delta({\bf 1}_{\{(x-x_0, x+x_0): x\in \Lambda(\epsilon)\}}+{\bf 1}_{\{(x+x_0, x-x_0): x\in \Lambda(\epsilon)\}}),
\end{eqnarray*}
where $\delta>0$ and ${\bf 1}_S$ is the indicator function of the set $S$. The function $\hat{g}$ is supported on the ``diagonal line'' $y=x$ and ``off-diagonal segments'' $\{(x-x_0, x+x_0): x\in \Lambda(\epsilon)\}$ and $\{(x+x_0, x-x_0): x\in \Lambda(\epsilon)\}$, which are disjoint for sufficiently small $\epsilon>0$. (This is similar to Figure 2). When $\delta>0$ is small enough, $\hat{g}(x, y)$ is non-negative everywhere. Furthermore, our construction preserves the ``total mass''. Hence, the function $\hat{g}(x, y)$ is indeed a probability density with respect to the $d$-dimensional Lebesgue measure on the ``diagonal line'' and two ``off-diagonal segments''. Let $(\hat{X}, \hat{Y})$ be a pair with the joint density $\hat{g}(x, y)$. The marginals $\hat{X}$ and $\hat{Y}$ have the same distribution as that of $X$, since the ``positive mass'' on ``off-diagonal points'' complements the ``mass deficit'' on ``diagonal points'' when we project in the $x$ and $y$ directions. We claim that $\frac{\hat{X}+\hat{Y}}{2}$ has larger entropy than $\hat{X}$. One can check that the density of $\frac{\hat{X}+\hat{Y}}{2}$ is
$$
\hat{f}(x)=f(x)+\delta(2{\bf 1}_{\Lambda(\epsilon)}-{\bf 1}_{\Lambda(\epsilon)+x_0}-{\bf 1}_{\Lambda(\epsilon)-x_0}).
$$
Let $\Omega$ denote the union of $\Lambda(\epsilon)$, $\Lambda(\epsilon)+x_0$ and $\Lambda(\epsilon)-x_0$. Then we have
\begin{equation}\label{eq:ent-avg}
h_r\left(\frac{\hat{X}+\hat{Y}}{2}\right)=\frac{1}{1-r}\log \left(\int_{\Omega}\hat{f}(x)^rdx+\int_{\Omega^c}f^r(x)dx\right).
\end{equation}
Since $x_0\neq0$, for $\epsilon>0$ small enough, $\Omega$ is the union of disjoint translates of $\Lambda(\epsilon)$. When $\delta>0$ is sufficiently small, we have
\begin{eqnarray}
\int_{\Omega}\hat{f}(x)^rdx
& = & \int_{\Lambda(\epsilon)} \left[(f(x)+2\delta)^r+(f(x+x_0)-\delta)^r+(f(x-x_0)-\delta)^r\right] dx \notag \\
& < & \int_{\Lambda(\epsilon)}\left[f(x)^r+f(x+x_0)^r+f(x-x_0)^r
\right]dx \label{eq:compare} \\
& = & \int_{\Omega}f(x)^rdx, \label{eq:ent-comp}
\end{eqnarray}
where inequality \eqref{eq:compare} follows from the observation that for $x\in\Lambda(\epsilon)\subset \Lambda$ (see \eqref{eq:lambda}) the derivative of the integrand at $\delta=0$ is
\begin{equation}
r[2f(x)^{r-1}-f(x-x_0)^{r-1}-f(x+x_0)^{r-1}]<0.
\end{equation}
Since $r>1$, \eqref{eq:ent-avg} together with \eqref{eq:ent-comp} implies that
$$
h_r\left(\frac{\hat{X}+\hat{Y}}{2}\right)>\frac{1}{1-r}\log \left(\int_{\Omega}f(x)^rdx+\int_{\Omega^c}f(x)^rdx\right)=h(X)=h(\hat{X}).
$$
This is contradictory to our assumption. Hence, $f$ has to be $(r-1)$-concave. For $1-1/d<r<1$, we redefine the set $\Lambda$ by reversing inequality \eqref{eq:lambda}, and inequality \eqref{eq:compare} will be also reversed. We will arrive at the same conclusion. 
\end{proof}

\begin{remark}
The proof of $1 \Longrightarrow 2$ is an immediate consequence of Theorem 3.36 in \cite{MMX17:1}. The theorem there draws heavily on the ideas of \cite{XMM16:isit}, where a related study, deriving the Schur convexity of R\'enyi entropies under the assumption of exchangeability and $s$-concavity of the random variables, generalizing Yu's results in \cite{Yu08:2} on the entropies of sums of i.i.d. log-concave random variables. Although we state Theorem \ref{thm:max-entropy} for two random vectors, the argument also works for more than two random vectors. Hence, it implies  the seemingly stronger Theorem \ref{thm: generalization of cover-zhang}.
\end{remark}

As an immediate consequence of Theorem \ref{thm:max-entropy}, we have the following reverse R\'enyi EPI for random vectors with the same distribution. 
\begin{corollary}
Let $s>-1/d$ and let $r=1+s$. Let $X$ and $Y$ be (possibly dependent) random vectors in $\R^d$ with the same density $f$ being $s$-concave. Then we have
$$
 N_r(X+Y) \leq 4 N_r(X). 
$$
\end{corollary}


\vspace{0.8cm}

\noindent Jiange Li \\
Einstein Institute of Mathematics \\
Hebrew University of Jerusalem \\
Jerusalem, IL 9190401 \\
E-mail: jiange.li@mail.huji.ac.il

\vspace{0.8cm}

\noindent Arnaud Marsiglietti \\
Department of Mathematics \\
University of Florida \\
Gainesville, FL 32611, USA \\
E-mail: a.marsiglietti@ufl.edu

\vspace{0.8cm}

\noindent James Melbourne \\
Electrical and Computer Engineering \\
University of Minnesota \\
Minneapolis, MN 55455, USA \\
E-mail: melbo013@umn.edu

\end{document}